\documentclass[11pt]{amsart} % use larger type; default would be 10pt
\usepackage{amsmath,amsfonts,amssymb}
\usepackage{amsthm}
\usepackage{xfrac}
\usepackage{hyperref}

\usepackage{color}

\DeclareMathOperator{\vol}{vol}
\DeclareMathOperator{\Out}{Out}
\DeclareMathOperator{\MCG}{MCG}
\DeclareMathOperator{\Hom}{Hom}

\theoremstyle{plain}
\newtheorem{thm}{Theorem}
\newtheorem{lem}[thm]{Lemma}
\newtheorem{prop}[thm]{Proposition}

\newtheorem{cor}[thm]{Corollary}
\newtheorem*{claim*}{Claim}
\newtheorem*{conj*}{Conjecture}

\theoremstyle{definition}
\newtheorem{defn}[thm]{Definition}

\newtheorem*{rmk*}{Remark}

\numberwithin{equation}{section}
\numberwithin{thm}{section}

\newcommand{\fakeenv}{} %%% prints the emptystring

\newenvironment{restate}[2]  %%% restate takes two arguments
{
  \renewcommand{\fakeenv}{#2} %%% So now \fakeenv prints #2
  \theoremstyle{plain}
  \newtheorem*{\fakeenv}{#1~\ref{#2}} %%% so now #2 is the name of a
  \begin{\fakeenv}
}
{
  \end{\fakeenv}
}

\usepackage[utf8]{inputenc} % set input encoding (not needed with XeLaTeX)

\usepackage{geometry} % to change the page dimensions
\usepackage{graphicx} % support the \includegraphics command and options

\title{\em Irreducibility of a free group endomorphism is a mapping torus invariant}
\author{Jean Pierre Mutanguha}

\address{Department of Mathematical Sciences \\ University of Arkansas \\ Fayetteville \\ AR \newline \indent  \it Web address: \tt \url{https://mutanguha.com/}}
\email{\href{mailto:jpmutang@uark.edu}{jpmutang@uark.edu}}

\date{} % Activate to display a given date or no date (if empty),
\begin{document}

\begin{abstract} We prove that the property of a free group endomorphism being irreducible is a group invariant of the ascending HNN extension it defines. This answers a question posed by Dowdall-Kapovich-Leininger. We further prove that being irreducible and atoroidal is a commensurability invariant. The invariance follows from an algebraic characterization of ascending HNN extensions that determines exactly when their defining endomorphisms are irreducible and atoroidal; specifically, we show that the endomorphism is irreducible and atoroidal if and only if the ascending HNN extension has no infinite index subgroups that are ascending HNN extensions.
\end{abstract}
\maketitle

\section{Introduction}

Suppose S is a hyperbolic surface of finite type and $f:S \to S$ is a {\it pseudo-Anosov} homeomorphism, then the interior of the {\it mapping torus} $M_f$ is a complete finite-volume hyperbolic 3-manifold; this is Thurston’s hyperbolization theorem for 3-manifolds that fiber over a circle \cite{Thu82}. It is a remarkable fact since the hypothesis is a dynamical statement about surface homeomorphisms but the conclusion is a geometric statement about 3-manifolds. In particular, since the converse holds as well, i.e., a hyperbolic 3-manifold that fibers over a circle will have a pseudo-Anosov {\it monodromy}, the property of a fibered manifold having a pseudo-Anosov monodromy is in fact a geometric invariant: if $f:S\to S$ and $f':S' \to S'$ are homeomorphisms whose mapping tori have {\it quasi-isometric} (q.i.) fundamental groups, then $f$ is pseudo-Anosov if and only if $f'$ is pseudo-Anosov.

There are three types of invariants that we study in geometric group theory: group invariants, which contain virtual/commensurability invariants, which contain geometric/q.i.-invariants; the geometric invariants are the most important and difficult to prove. In this paper, we exhibit geometric and commensurability invariants for free-by-cyclic groups inspired by Thurston's hyperbolization theorem and our arguments will be general enough to also apply to ascending HNN extensions of free groups.

There is a rough correspondence between the study of the outer automorphism group of a finitely generated free group $\Out(F)$ and the study of the mapping class group of a hyperbolic surface of finite type $\MCG(S)$. Under this correspondence, surface groups are paired with free groups, surfaces with graphs, and 3-manifolds that fiber over a circle with free-by-cyclic groups.
However, this correspondence is not perfect; pseudo-Anosov mapping classes have three possible analogous properties for free group automorphisms: {\it induced by a pseudo-Anosov (on a punctured surface)}, {\it atoroidal}, and {\it irreducible} (Section~\ref{defs}).
We originally set out to prove that irreducibility was a group invariant of the automorphism's mapping torus and, along the way, we proved more general statements for the first property and the composite property of being both irreducible and atoroidal. Our first result is that the first property is a geometric invariant:

\begin{restate}{Theorem}{geomqi} Suppose $\phi: F \to F$ and $\psi:F' \to F'$ are injective endomorphisms of free groups such that the mapping tori $F *_\phi$ and $F' *_\psi$ are quasi-isometric. Then $\phi$ is induced by a pseudo-Anosov if and only if $\psi$ is induced by a pseudo-Anosov.
 \end{restate}

Thus starting with just a free group automorphism $\phi$ induced by a pseudo-Anosov and a quasi-isometry between $F \rtimes_\phi \mathbb Z$ and $F' *_\psi$, we find that $\psi$ is induced by a surface homeomorphism too. The proof is short but uses deep geometric theorems: Thurston's hyperbolization \cite{Thu82} and Schwartz' rigidity \cite{Sch95}. Since pseudo-Anosovs have dynamics that are very similar to those of irreducible and atoroidal automorphisms, it is likely that the latter property is a geometric invariant too.

\begin{conj*}[q.i.-invariance of irreducibility] Suppose $\phi : F \to F$ and $\psi : F' \to F'$ are automorphisms of free groups such that $F \rtimes_\phi \mathbb Z$ and $F' \rtimes_\psi \mathbb Z$ are quasi-isometric. Then $\phi$ is irreducible and atoroidal if and only if $\psi$ is irreducible and atoroidal.
\end{conj*}

Our main result is that being irreducible and atoroidal is a commensurability invariant, which lends credence to the conjecture; again, the argument works for endomorphisms.

\begin{restate}{Theorem}{nongeomcomm} Suppose $\phi: F \to F$ and $\psi:F' \to F'$ are injective endomorphisms of free groups such that $F *_\phi$ and $F' *_\psi$ are commensurable and neither endomorphism has an image contained in a proper free factor of its codomain. Then $\phi$ is irreducible and atoroidal if and only if $\psi$ is irreducible and atoroidal.
 \end{restate}

\begin{rmk*} The hypothesis on the images is necessary. Let $\phi:F_2\to F_2$ be the endomorphism on a free group of rank $2$ given by $\phi(a) = ab$ and $\phi(b) = ba$. Then $\phi$ is irreducible and atoroidal \cite[Example~1.3]{JPM}. Now let $F_2$ be a proper free factor of the free group $F_3$ generated by $\{ a,b,c \}$.
Extend $\phi$ to $\psi:F_3 \to F_3$ by setting $\psi(c) \in F_2$; then $F_3*_{\psi} \cong F_2*_\phi$, but $\psi$ is reducible.\end{rmk*}

The proof of Theorem~\ref{nongeomcomm} follows immediately from an algebraic characterization of $F*_\phi$ that detects exactly when $\phi: F \to F$ is irreducible and atoroidal.

\begin{restate}{Theorem}{grpinv} Let $\phi: F \to F$ be an injective endomorphism of a free group whose image is not contained in a proper free factor of $F$. Then $\phi$ is irreducible and atoroidal if and only if any finitely generated noncyclic subgroup $H \le F*_\phi$ with Euler characteristic $\chi(H)=0$ has finite index. \end{restate}

These results imply that irreducibility is a group invariant, our original motivation:

\begin{restate}{Corollary}{irredgrp} Suppose $\phi: F \to F$ and $\psi:F' \to F'$ are injective endomorphisms of free groups such that $F *_\phi \cong F' *_\psi$ and neither one of the endomorphisms has an image contained in a proper free factor. Then $\phi$ is irreducible if and only if $\psi$ is irreducible.
 \end{restate}

That irreducibility is a group invariant was an open problem \cite[Question~1.4]{DKL17c}. In a series of papers \cite{DKL15, DKL17b, DKL17c}, Dowdall-Kapovich-Leininger studied the dynamics of ({\it word-hyperbolic}) free-by-cyclic groups and the main result of the third paper answered this problem under an extra condition that we now discuss:

Fix a free-by-cyclic group $G$. The {\it BNS-invariant} $\mathcal C(G)$ is an open cone (with convex components) in $H^1(G; \mathbb R) \cong \Hom(G, \mathbb R)$. By rational rays in $H^1(G; \mathbb R)$, we refer to projective classes of homomorphisms $G \to \mathbb R$ with discrete/cyclic image. Without defining the BNS-invariant, we shall state its most relevant property for our purposes: a rational ray in $H^1(G; \mathbb R)$ is {\it symmetric}, i.e., is in $- \mathcal C(G) \cap \mathcal C(G)$ if and only if the corresponding class of homomorphisms $[p] : G \to \mathbb R$ has finitely generated kernel $K$; in this case, $ K $
is free for cohomological reasons \cite{FH99, Bie81, St68}, $G \cong K \rtimes_{\phi}  \mathbb Z$ for some free group automorphism $\phi: K \to K$, and the natural projection $K \rtimes_{\phi} \mathbb Z \to \mathbb Z$ is in the projective class $[p]$. Fix a symmetric rational ray $r_0$ in $\mathcal C(G)$, and let $\phi_0: K_0 \to K_0$ be the corresponding free group automorphism. The presentation complex for $K_0 \rtimes_{\phi_0} \mathbb Z$
has a natural semi-flow with respect to the {\it stable direction} $\mathbb Z_+$. Dowdall-Kapovich-Leininger show in \cite{DKL17b} that getting from $r_0$ to any symmetric rational ray in the same component of $\mathcal C(G)$ amounts to reparametrizing this semi-flow (hence the convexity of the component) and with a careful analysis of this semi-flow, they are able to relate the monodromy stretch factors and rank of kernels for all symmetric rays in the same component. In the third paper \cite{DKL17c}, they conclude this analysis by showing that being irreducible and atoroidal is an invariant of a component of the BNS-invariant; that is, if $\phi_1: K_1 \to K_1$ and $\phi_2: K_2 \to K_2$ are free group automorphisms corresponding to symmetric rays in the same component of $\mathcal C(G)$, then $\phi_1$ is irreducible and atoroidal if and only if $ \phi_2 $ is too. Since this result relied heavily on the analysis of the semi-flow for a component of the BNS-invariant, the technique cannot be extended to work for symmetric rational rays in separate components. Furthermore, their result does not apply to any asymmetric rational rays, i.e., it does not apply to nonsurjective injective endomorphisms. Theorem~\ref{nongeomcomm} addresses both of these concerns.

Masai-Mineyama have also proven a different special case of Theorem~\ref{nongeomcomm} that they call {\it fibered commensurability} \cite{MM17}:
suppose $\phi: F \to F$ and $\psi: F' \to F'$ are free group automorphisms and let $K \le F, K' \le F'$ be finite index subgroups with an isomorphism $\epsilon: K \to K'$ such that the outer classes of the isomorphisms $\epsilon\left.\phi^i\right|_{K}: K \to K'$ and $\left.\psi^j \epsilon\right|_{K}:K \to K'$ are the same for some $i, j \ge 1$; a priori, we restrict ourselves to $i, j$ such that
$\phi^i(K) = K$ and $\psi^j(K') = K'$, i.e., the maps $\epsilon\left.\phi^i\right|_{K}$ and $\left.\psi^j \epsilon\right|_{K}$ make sense. Masai-Mineyama prove that in this case, $\phi$ is irreducible and atoroidal if and only if $\psi$ is too. In other words, being irreducible and atoroidal is a fibered commensurability invariant. However, compared to commensurability, this equivalence is very restrictive since a typical isomorphism of finite index subgroups of free-by-cylic groups need not preserve the fibers one starts with.

Feighn-Handel's theorem that mapping tori of free group endomorphisms are coherent is the main tool that allows us to avoid the obstacles in the two approaches discussed above. We shall explicitly use the {\it preferred presentation} that they prove exists for any finitely generated subgroup of a mapping torus (Theorem~\ref{rewrite}).

We conclude the introduction by noting that little is known about geometric invariants of free-by-cyclic groups in relation to their monodromies. Here are the geometric invariants we know of: Brinkmann \cite{Bri00} showed a free-by-cyclic group is word-hyperbolic if and only if the monodromy is atoroidal; Kapovich-Kleiner \cite[Proof of Corollary 15]{KK00} characterize the {\it Gromov boundary} of a word-hyperbolic free-by-cyclic group: it is a Menger curve if and only if the monodromy fixes no free splitting of the fiber (See also~\cite{Bri02}); Macura \cite{Mac02} proved that two polynomially-growing free group automorphisms that have quasi-isometric mapping tori must have polynomial growth of the same degree and, conjecturally, these mapping tori cannot be quasi-isometric to a mapping torus of an exponentially growing free group automorphism.
\newline

\noindent {\bf Outline.} We give the standard definitions and statements of results that are most relevant to the rest of the article in Section~\ref{defs}. Section~\ref{geom} contains the proof of Theorem~\ref{geomqi} while Section~\ref{nongeom} contains that of Theorem~\ref{nongeomcomm}. Finally, we briefly discuss the q.i.-invariance conjecture in Section~\ref{qi}. In Appendix~\ref{app}, we prove a folk theorem and its converse: a free group endomorphism is irreducible with a nontrivial periodic conjugacy class if and only if it is induced by a pseudo-Anosov on a punctured surface with one orbit of punctures.
\newline

\noindent \textbf{Acknowledgments:} I want to thank Ilya Kapovich for patiently checking my initial proof (and its numerous iterations) and Chris Leininger for sharing with me his work-in-progress in collaboration with Spencer Dowdall and I. Kapovich. The appendix was born out of an insightful discussion with Saul Schleimer on a bus ride and it would not have been written without Mladen Bestvina's encouragement. Last but not least, I thank my advisor Matt Clay for his constant support.

%%%%%%%%%%%%%%%%%%%%%%%%%%%%%%%%%%%%%%%%%%%%%%
\section{Definitions and Preliminaries}\label{defs}

In this paper, $F$ and $F'$ are free groups with finite ranks at least $2$. We will study the ascending HNN extension of a free group and how its properties relate to those of the defining endomorphism. Let $A \le F$ be a subgroup of a free group and $\phi : A \to F$ be an injective
homomorphism, then we define the {\bf HNN extension} of $F$ over $A$ to be:
\[ F*_A = \left\langle\, F, t~|~t^{-1} a t = \phi(a), \forall a \in A \,\right\rangle \]
An HNN extension has a natural map  $F*_A \to \mathbb Z$ that maps $F \mapsto 0$ and $ t \mapsto 1$; we shall refer to this map as the {\bf natural fibration}. For the rest of this paper, we restrict ourselves to HNN extensions defined over free factors. When $A = F$, then we call $F*_F = F*_\phi$ an {\bf ascending HNN extension} or a {\bf mapping torus} of $\phi : F \to F$. The latter terminology stems from the fact that the injective endomorphism $\phi$ can be topologically represented by a graph map on the rose whose topological mapping torus has a fundamental group isomorphic to $F*_\phi$ . Following this analogy, we shall call $F$ the {\bf fiber} and $\phi$ the {\bf monodromy} of the mapping torus. Finally, when $\phi : F \to F$ is an automorphism, we call $F*_\phi = F \rtimes_\phi \mathbb Z$ a {\bf free-by-cyclic group}.
\newline

The following are the properties of monodromies that we will study. An endomorphism $\phi: F \to F$ is {\bf reducible} if there is a free factorization $A_1 * \cdots * A_k * B $ of $F$ where $B$ may be trivial if $k \ge 2$, and a sequence of elements, $(x_i)_{i=1}^k$,  in $F$ such that $\phi(A_i) \le x_i A_{i+1} x_i^{-1}$ where the indices are considered$\mod k$; the collection $\{A_1, \ldots, A_k \}$ is a {\bf $\boldsymbol{\phi}$-invariant proper free factor system}.
An endomorphism $\phi$ is {\bf irreducible} if it is not reducible. It {\bf has a nontrivial periodic conjugacy class} if there is a nontrivial element $a \in F$, element $x \in F$, and integer $n \ge 1$ such that $\phi^n(a) = x a x^{-1}$, equivalently, $i_x \phi^n(a) = a$ for some inner automorphism~$i_x$. Otherwise, it is {\bf atoroidal}. Finally, $\phi$ has {\bf finite-order} if $\phi^n$ is an inner automorphism for some integer $n \ge 1$.
\begin{rmk*} In the literature, the property {\bf fully irreducible}, also known as {\bf irreducible with irreducible powers (iwip)}, is more prevalent than irreducible. It stands for irreducible endomorphisms whose iterates are all irreducible too and being closed under iteration makes it a more natural dynamical property. It was previously known that atoroidal irreducible endomorphisms were fully irreducible \cite{DKL15, JPM} and some infinite-order irreducible endomorphisms were not fully irreducible \cite{BH92}. In Appendix~\ref{app}, we characterize infinite-order irreducible endomorphisms that are not fully irreducible (Corollary~\ref{irredIWIP}).
\end{rmk*}

Feighn-Handel used the following proposition to show the coherence of ascending HNN extensions of free groups; the proposition and the next lemma allow us to rewrite presentations of ascending HNN extension subgroups of ascending HNN extension groups so that fibers of the subgroup lie in fibers of the ambient group.

\begin{thm}[{\cite[Proposition 2.3]{FH99}}]\label{rewrite} Suppose $\Phi:\mathbb F \to \mathbb F$ is an injective endomorphism of a free group (of possibly infinite rank) and $G = \mathbb F*_\Phi = \left\langle\, \mathbb F, t~|~t^{-1} x t = \Phi(x), \forall x\in\mathbb F \,\right\rangle$.
Let $p: G \to \mathbb Z$ be the natural fibration that maps $\mathbb F \mapsto 0$ and $t \mapsto 1$.

If $H \le G$ is finitely generated and $\left.p\right|_H$ is not trivial and not injective, then there is an isomorphism $\iota: F*_A \to H$, natural fibration $\pi: F*_A \to \mathbb Z$, and injective map $n:\mathbb Z \to \mathbb Z$ such that $n(1) \ge 1$ and $\left.p\right|_H \iota = n \pi$, where $F*_A$ is an HNN extension of a finitely generated free group $F$ over a free factor $A \le F$.
\end{thm}

As a consequence of this theorem, all finitely generated subgroups of an ascending HNN extension of a free group are finitely presented with presentation complexes that have contractible universal covers. In particular, we can define their Euler characteristic. The next lemma will be used to show certain subgroups of a mapping torus are also mapping tori if (and only if) they have Euler characteristic zero.

\begin{lem}\label{hwlem} Let $F$ be a free group of finite rank $m$ and $A \le F$ a free factor of rank $n$. Then $\chi(F*_{A}) = n - m$.
\end{lem}
\begin{proof} Choose a basis for $A$ and extend it to a basis for $F$. Then the natural finite presentation complex $X$ is aspherical, i.e., $X$ is a $K(F*_A,1)$ space, and \[\chi(F*_{A}) = \chi(X) = 1 - (1+m) + n = n-m. \qedhere\]
\end{proof}

Finally, we give a characterization of irreducible endomorphisms with nontrivial periodic conjugacy classes. As the techniques used in the proof of the following theorem are not related to the rest of the paper, we postpone the proof to the appendix (Proposition~\ref{eg} and Theorem~\ref{thmBH2}).

\begin{thm}\label{thmBH} Let $\phi:F \to F$  be an infinite-order endomorphism.
Then $\phi$ is irreducible and has a periodic nontrivial conjugacy class if and only if it is induced by a pseudo-Anosov homeomorphism of a compact surface $\Sigma_g^b$ that acts transitively on the boundary components.
\end{thm}

This theorem allows us to partition injective endomorphisms of free groups with interesting dynamics into two categories:
\begin{itemize}
\item automorphisms induced by pseudo-Anosov maps.
\item irreducible and atoroidal endomorphisms.
\end{itemize}

%%%%%%%%%%%%%%%%%%%%%%%%%%%%%%%%%%%%%%%%%%%%%
\section{Pseudo-Anosov Monodromies}\label{geom}

The first result in this section is a straightforward application of Stallings' fibration theorem and Nielsen-Thurston classification; it is the first half of the proof that irreducibility of the monodromy is a group invariant of the mapping torus. The subsequent generalizations are given to motivate the q.i.-invariance conjecture.

\begin{prop}\label{geomgrp} Suppose $\phi: F \to F$ and $\psi:F' \to F'$ are injective endomorphisms such that $F *_\phi \cong F' *_\psi$. Then $\phi$ is induced by a pseudo-Anosov if and only if $\psi$ is induced by a pseudo-Anosov.\end{prop}
\begin{proof} Without loss of generality, suppose $\phi$ is induced by a pseudo-Anosov on a compact surface with boundary. Therefore, $G = F \rtimes_\phi \mathbb Z$ is the fundamental group of a compact $3$-manifold that fibers over a circle. Bieri-Neumann-Strebel \cite{BNS} showed that the BNS-invariant of such a fundamental group is symmetric, which implies that $\psi$ is an automorphism and $G \cong F' \rtimes_\psi \mathbb Z$. By Stallings' fibration theorem \cite{St61}, $\psi$ is induced by a homeomorphism of a compact surface with boundary.
Any invariant essential multicurve of the $\psi$-inducing homeomorphism would determine a non-peripheral $\mathbb Z^2$ subgroup of $G$. But since $\phi$ was induced by a pseudo-Anosov, the only $\mathbb Z^2$ subgroups of $G$ are the peripheral ones. Similarly, $\psi$ cannot have finite-order. Thus $\psi$ is induced by an {\it infinite-order irreducible} homeomorphism. By Nielsen-Thurston classification, $\psi$ is induced by a pseudo-Anosov.
\end{proof}

Since the number of orbits of boundary components for a surface homeomorphism is the number of boundary components in the mapping torus, this proposition combines with Theorem~\ref{thmBH} to give:

\begin{cor}\label{irredtorgrp} Suppose $\phi: F \to F$ and $\psi:F' \to F'$ are injective endomorphisms such that $F *_\phi \cong F' *_\psi$. Then $\phi$ is irreducible and has a periodic nontrivial conjugacy class if and only if $\psi$ is irreducible and has a periodic nontrivial conjugacy class.
\end{cor}

Surprisingly, the analogous statement for commensurable groups is much harder to prove. The difficulty lies in showing that if the restriction of an iterate $\psi^n$ to a finite index subgroup is induced by a pseudo-Anosov, then so is $\psi$. We could use the theory of train tracks to adapt the argument in Appendix~\ref{app} and get a comparatively elementary proof; we opt to use Thurston's hyperbolization \cite{Thu82} and Mostow's rigidity \cite{Mar74} to keep the exposition short.

\begin{prop}\label{geomcomm} Suppose $\phi: F \to F$ and $\psi:F' \to F'$ are injective endomorphisms such that $F *_\phi$ and $F' *_\psi$ are commensurable. Then $\phi$ is induced by a pseudo-Anosov if and only if $\psi$ is induced by a pseudo-Anosov.\end{prop}
\begin{proof} Let $H \le F*_\phi$ and $H' \le F'*_\psi$ be isomorphic finite index subgroups. Without loss of generality, suppose $\phi$ is induced by a pseudo-Anosov homeomorphism on a compact surface with boundary. 
By Thurston's hyperbolization theorem,  $F \rtimes_\phi \mathbb Z$, and hence $H' \cong H$, is a fundamental group of compact 3-manifold whose interior $M$ has a complete finite-volume hyperbolic structure. Assume $\pi_1(M) = H' \trianglelefteq F' *_\psi$, then Mostow's rigidity implies the finite group $Q = (F' *_\psi)/H'$ acts on $M$ by isometries. The action is free since $F' *_\psi$ is torsion-free. Thus $F' *_\psi$ is the fundamental group of the cusped hyperbolic 3-manifold $N=M/Q$; we may identify $N$ with the interior of a compact manifold $\bar{N}$ while maintaining $\pi_1(\bar N) \cong F' *_\psi$. By symmetry of the BNS-invariant and Stallings' fibration theorem, we deduce $N$ is a hyperbolic mapping torus whose associated monodromy is a pseudo-Anosov that induces $\psi$.
\end{proof}

We mentioned before the proposition that the geometric proof can be replaced by an argument using the theory of train tracks. However, the next theorem is a geometric statement and there is no apparent way around Thurston’s hyperbolization theorem. In fact, we will also need Schwartz’ rigidity \cite{Sch95} to reduce the theorem to the previous proposition.

\begin{thm}\label{geomqi}Suppose $\phi: F \to F$ and $\psi:F' \to F'$ are injective endomorphisms such that $F *_\phi$ and $F' *_\psi$ are quasi-isometric. Then $\phi$ is induced by a pseudo-Anosov if and only if $\psi$ is induced by a pseudo-Anosov.\end{thm}
\begin{proof} Suppose $\phi$ is induced by a pseudo-Anosov. By Thurston's hyperbolization theorem again, $F*_\phi$ is the fundamental group of a complete finite-volume hyperbolic 3-manifold with cusps. In particular, Richard Schwartz proved such groups are {\it q.i.-rigid}. As $F*_\phi$ and $F'*_\psi$ are quasi-isometric torsion-free groups, q.i.-rigidity of $F*_\phi$ implies they are commensurable. Thus $\psi$ is induced by a pseudo-Anosov by Proposition~\ref{geomcomm}.
\end{proof}

This proof underscores how difficult it is to prove the q.i.-invariance conjecture since there is no common model like $\mathbb H^3$ when studying irreducible and atoroidal endomorphisms.

%%%%%%%%%%%%%%%%%%%%%%%%%%%%%%%%%%%%%%%%%%%%%
\section{Irreducible and Atoroidal Monodromies}\label{nongeom}

The goal of this section is to prove that being irreducible and atoroidal is a commensurability invariant. We proved the following result in previous work \cite{JPM}; it essentially characterizes being an irreducible and atoroidal endomorphism.

\begin{prop}[{\cite[Proposition 5.4]{JPM}}]\label{invSbgrp} Suppose $\phi : F \to F$ is an irreducible and atoroidal endomorphism. If nontrivial $K \le F$ is finitely generated and $i_x\phi^n(K) \le K$ for some $n \ge 1$ and inner automorphism $i_x$, then $[F : (i_x \phi^n)^{-k}(K)] < \infty$ for some $k \ge 0$.
\end{prop}

The key idea in this section is to use this proposition to characterize in terms of the mapping torus exactly when a monodromy is irreducible and atoroidal. To this end, we need the following property to deal with nonsurjective monodromies.

\begin{defn} An injective endomorphism $\phi:F \to F$ is {\bf minimal} if its image $\phi(F)$ is not contained in a proper free factor of $F$. Automorphisms and irreducible endomorphisms are clearly minimal.\end{defn}

Minimality is closed under composition (See \cite[Proof of Theorem 5.5]{JPM}). We now have enough to state and prove the main result:

\begin{thm}\label{grpinv} Suppose $\phi: F \to F$ is a minimal injective endomorphism and let $G = F*_\phi$. Then $\phi$ is irreducible and atoroidal if and only if any finitely generated noncyclic subgroup of $H \le G$ with Euler characteristic $\chi(H) = 0$ has finite index.\end{thm}

\begin{proof} If $\phi$ is not atoroidal, then $G$ has a $\mathbb Z^2$ subgroup that necessarily has infinite index as $F$ is not cyclic.
If $\phi$ is reducible, then there exists a proper free factor $A \le F$, $x \in F$, and $n \ge 1$ such that $\phi^n(A) \le xAx^{-1}$. Then, using normal forms, $A *_{i_x \phi^n} \cong \langle A, t^n x \rangle \le G$. Suppose $[G: \langle A, t^n x \rangle] < \infty$, then $[F: F \cap \langle A, t^n x \rangle] < \infty$.
Set $K = F \cap \langle A, t^n x \rangle = \cup_k(i_x \phi^n)^{-k}(A)$. As $[F: K] < \infty$, $K$ is finitely generated and there exists a $k_0 \ge 1$ such that $K = (i_x \phi^n)^{-k_0}(A)$. The statements $[F:K] < \infty$, $K = (i_x \phi^n)^{-k_0}(A)$, and $A$ is a proper free factor of $F$ imply $F = (i_x \phi^n)^{-k_0}(A)$, which contradicts the minimal assumption on $\phi$.
Therefore, $[G: \langle A, t^n x \rangle] = \infty$ as needed. This concludes the reverse direction.

For the forward direction, suppose $\phi$ is irreducible and atoroidal and let $H \le G$ be a finitely generated noncyclic group with $\chi(H) = 0$. We need to show $[G: H] < \infty$. Let $p:G=F*_\phi \to \mathbb Z$ be the natural fibration. Note that $\ker p = \cup_i t^i F t^{-i}$. Then $\left.p\right|_H$ is not trivial since $\ker p$ is locally free yet $H$ is finitely generated but not free: it is not cyclic and $\chi(H) = 0$. Also $\left.\ker p\right|_H = H \cap \ker p$ is not trivial as $H \not \cong \mathbb Z$.

As $H$ is finitely generated and $\left.p\right|_H$ is not trivial and not injective, by Proposition~\ref{rewrite}, there is an isomorphism $\iota: F_m*_A \to H$, natural fibration $\pi: F_m*_A \to \mathbb Z$, and injective map $n:\mathbb Z \to \mathbb Z$ such that $n(1) \ge 1$ and $\left.p\right|_H \iota = n \pi$, where $F_m*_A$ is an HNN extension of a finitely generated free group $F_m$ over a free factor $A \le F_m$.
As $\chi(H) = 0$, $A$ is not a proper free factor by Lemma~\ref{hwlem}. Therefore, $H \cong F_m*_{F_m}$.
As $F_m$ is finitely generated, there is an $i \ge 0$ such that $K = \iota(F_m) \le t^i F t^{-i} \le \ker p$. Fix large enough~$i$, then $K \le t^i F t^{-i}$ is a finitely generated nontrivial subgroup such that $x \in t^i F t^{-i}$ and $i_x\bar\phi^n(K) \le K$, where $n = n(1)$ and $\bar \phi$ is the natural extension of $\phi$ to $t^i F t^{-i}$.
As $\phi$ is irreducible and atoroidal, so is~$\bar \phi$. By Proposition~\ref{invSbgrp}, $(i_x \bar\phi^n)^{-k}(K)$ has finite index in $t^i F t^{-i}$ for some $k \ge 0$. Therefore, $H = \langle (t^n x)^k K (t^n x)^{-k}, t^n x \rangle$ has finite index in $G = \langle t^i F t^{-i}, t \rangle$.
\end{proof}

The next lemma shows that having a finitely generated noncyclic subgroup with infinite index and Euler characteristic zero is a commensurability invariant.

\begin{lem}Let $G$ be a finitely generated torsion-free group and $G' \le G$ be a finite index subgroup. Then $G$ has an infinite index finitely generated noncyclic subgroup $H$ with Euler characteristic $\chi(H) = 0$ if and only if $G'$ has such a subgroup.
\end{lem}
\begin{proof}
The reverse direction is obvious. Suppose $H \le G$ is an infinite index finitely generated noncyclic subgroup with Euler characteristic $\chi(H) = 0$. Then $H \cap G'$ has finite index in $H$. In particular, it is finitely generated with infinite index in $G$ and hence $G'$. Furthermore, $[ H: H \cap G'] < \infty$ implies $\chi(H \cap G') = [H: H \cap G'] \cdot \chi(H) = 0$ and $H \cap G'$ is noncyclic since the only virtually cyclic torsion-free group is $\mathbb Z$, yet $H$ is not cyclic.
\end{proof}

This implies the invariance of irreducibility for atoroidal endomorphisms.

\begin{thm}\label{nongeomcomm}Suppose $\phi: F \to F$ and $\psi:F' \to F'$ are minimal injective endomorphisms such that $F *_\phi$ and $F' *_\psi$ are commensurable. Then $\phi$ is irreducible and atoroidal if and only if $\psi$ is irreducible and atoroidal.\end{thm}

Finally, we combine this theorem with Corollary~\ref{irredtorgrp} to get:
\begin{cor}\label{irredgrp}Suppose $\phi: F \to F$ and $\psi:F' \to F'$ are minimal injective endomorphisms such that $F *_\phi \cong F' *_\psi$. Then $\phi$ is irreducible if and only if $\psi$ is irreducible.\end{cor}

%%%%%%%%%%%%%%%%%%%%%%%%%%%%%%%%%%%%%%%%%%%%%
\section{Q.I.-Invariance Conjecture}\label{qi}

Little is known about the geometry of mapping tori whose monodromies are irreducible and atoroidal. Brinkmann \cite{Bri00} proved that an atoroidal automorphism of a free group has a word-hyperbolic mapping torus and Kapovich-Kleiner \cite[Corollary 15]{KK00} show the following are equivalent for such a mapping torus:
\begin{itemize}
  \item it has (Gromov) boundary homeomorphic to a Menger curve;
  \item it does not split over cyclic subgroups;
  \item the monodromy does not fix a {\it free splitting} of $F$. (See also \cite{Bri02})
\end{itemize}

Unfortunately, the topology of the boundary cannot detect irreducible atoroidal monodromies as there are reducible atoroidal automorphisms that do not fix a free splitting.
The boundary of a word-hyperbolic group can be given a {\it visual metric} that is unique up to {\it quasi-symmetry} and this quasi-symmetry class is a complete geometric invariant of the group, i.e., word-hyperbolic groups are quasi-isometric if and only if their visual boundaries are quasi-symmetric \cite{Pau96}. Although a reducible and an irreducible automorphism can have word-hyperbolic mapping tori with homeomorphic boundaries, the conjecture asserts that the visual boundaries cannot be quasi-symmetric!

%%%%%%%%%%%%%%%%%%%%%%%%%%%%%%%%%%%%%%%%%%%%%
\appendix
\section{}\label{app}

The following is a folk theorem that was used to construct examples of fully irreducible automorphisms along with examples that are infinite-order irreducible but not fully irreducible \cite[Example 1.4]{BH92}. At the end of the appendix, we will show that the latter examples are complete/exhaustive (Corollary~\ref{irredIWIP}).

\begin{prop}\label{eg} If $\phi:F \to F$ is an automorphism induced by a pseudo-Anosov homeomorphism of a compact surface $\Sigma_g^{b \ge 1}$ that acts transitively on the boundary components, then $\phi$ is irreducible and it is fully irreducible if and only if $b = 1$.
\end{prop}
\begin{proof} Let $f : \Sigma \to \Sigma$ be the inducing pseudo-Anosov homeomorphism.
Suppose $\phi$ was reducible, i.e., there exists a $\phi$-invariant proper free factor system $\{ A_1, \ldots, A_k\}$ of $F$. Thus $i_j \phi^k(A_j) \le A_j$ for some inner automorphisms $i_j$. Let $\hat \Sigma\to \Sigma$ be the cover corresponding to $A_j \le F \cong \pi_1(\Sigma)$;
the previous inclusion implies $f^k$ lifts to a homeomorphism $\hat f: \hat \Sigma \to \hat \Sigma$. Furthermore, up to isotopy, $\hat f$ preserves the core of $\hat \Sigma$, a compact subsurface that supports $A_j$. Let $\gamma \subset \hat \Sigma$ be a boundary component of the core. After replacing $f$ with a power, we can assume $\gamma$ is an $\hat f$-fixed simple closed curve. However, the projection of $\gamma$ to $\Sigma$ may not be a simple closed curve.
Using the {\it LERF} property of $F \cong \pi_1(\Sigma)$, construct a finite cover $\bar \Sigma$ such that the projection of $\gamma$ to $\Sigma$ lifts to a simple closed curve $\bar \gamma$ in $\bar \Sigma$.

Since $\phi$ is an automorphism, after passing to a power, we can assume $f$ lifts to a homeomorphism $\bar f: \bar \Sigma \to \bar \Sigma$. This map is pseudo-Anosov since the $f$-invariant measured foliations on $\Sigma$ lift to $\bar f$-invariant measured foliations on $\bar \Sigma$.  After passing to a power again, we can assume $\bar \gamma$ is $\bar f$-fixed. But the periodic nontrivial simple closed curves of a pseudo-Anosov homemorphism are exactly the peripheral curves. Thus, the projection of $\gamma$ to $\Sigma$ is a (power of a) peripheral curve. Since this holds for any boundary component $\gamma$ of the core of $\hat \Sigma$, it must be that $A_j$ is a cyclic free factor corresponding to a boundary component of $\Sigma$.
But $f$ acts transitively on the boundary components, hence there is a one-to-one correpondence between the cyclic free factors $A_1, \ldots, A_k$ and the boundary components of $\Sigma$. This is a contradiction, as all boundary components of a compact surface with boundary cannot be simultaneously realized as free factors of its fundamental group. Therefore, $\phi$ is irreducible.

If $b=1$, then all powers of $f$ are pseudo-Anosovs that act transitively on the boundary component. Thus, all powers of $\phi$ are irreducible.
If $b \ge 2$, then any boundary component of $\Sigma$ determines a periodic proper free factor and some power of $\phi$ is reducible.\end{proof}

For the rest of the appendix, we will prove the converse. Bestvina-Handel use the following proposition to prove that fully irreducible automorphisms with periodic nontrivial conjugacy classes are induced by pseudo-Anosov homeomorphisms of surfaces with one boundary component \cite[Proposition 4.5]{BH92}.

\begin{prop}[{\cite[Lemma 3.9]{BH92}}] Suppose an automorphism $\phi:F \to F$ is fully irreducible and there exists $k \ge 1$ and nontrivial conjugacy class $[c]$ of $F$ such that $[\phi^k(c)] = [c]$.

Then some iterate of $\phi$ has an irreducible train track representative with exactly one (unoriented) indivisible Nielsen loop which covers each edge of the graph twice.

In particular, if $k$ is minimal, then $k \le 2$ with equality if and only if $[\phi(c)] = [c^{-1}]$.
\end{prop}

We start by extending this proposition to irreducible endomorphisms. For brevity, we assume familiarity with Bestvina-Handel's argument \cite[Section 3]{BH92}. The main change is we work with periodic indivisible Nielsen paths (of a fixed period) rather than indivisible Nielsen paths. The argument in the final paragraph is different too: unlike Bestvina-Handel's argument where periodic proper free factors are enough to contradict full irreducibility, we construct a fixed proper free factor system to contradict irreducibility.

\begin{prop}\label{propStab} Suppose an endomorphism $\phi:F \to F$ is irreducible, it has infinite-order, and there exists $k \ge 1$ and nontrivial conjugacy class $[c]$ of $F$ such that $[\phi^k(c)] = [c]$.

Then there exists an irreducible train track representative $f:\Gamma \to \Gamma$ with exactly one $f$-orbit of (unoriented) periodic indivisible Nielsen paths that make up an $f$-orbit of (unoriented) periodic Nielsen loops that collectively cover each edge of $\Gamma$ twice.
\end{prop}
\begin{proof}[Sketch proof]
Represent $\phi$ by an expanding irreducible train track map $f:\Gamma \to \Gamma$. Let $\sigma$ be a loop representing $[c]$. By hypothesis, $f^k(\sigma) \simeq \sigma$. Break $\sigma$ into maximal legal segments $\sigma_0, \ldots, \sigma_{s-1}$, and note that, since $f$ is a train track, tightening the loop $f^k(\sigma)$ produces a cyclic permutation of the maximal legal segments $\sigma_i, \ldots \sigma_{s-1+i}$ for some $i$. Then $f^{ks}(\sigma)$ will give the original maximal legal segments $\sigma_0, \ldots, \sigma_s$. In particular, as $f$ is expanding, each segment $\sigma_j$ has a $f^{ks}$-fixed point and $\sigma$ is an $f^{ks}$-Nielsen loop, or equivalently, $\sigma$ is an $f$-periodic Nielsen loop.

The goal is to replace $f$ with another train track representative such that the $f$-orbit of $\sigma$ is a set of (unoriented) loops that collectively cover each edge of $\Gamma$ twice.

For the rest of the proof, a periodic indivisible Nielsen path (piNp) will refer to an indivisible Nielsen path of the fixed $ks$-iterate. It follows from the bounded cancellation lemma that there are finitely many piNps.

The $f$-orbit of $\sigma$ breaks into piNps; just like indivisible Nielsen paths, a piNp $\rho$ can be uniquely written as a concatenation of two legal paths $\rho = \alpha \beta$. By an $f$-orbit of unoriented piNps, we mean the minimal sequence $\{ \rho_0 = \rho, \rho_1 = f(\rho), \ldots, \rho_m = f^m(\rho) \}$ such that $f^{m+1}(\rho) = \rho \text{ or } \bar \rho$ for some piNp $\rho$.
If we write each $\rho_i$ in the $f$-orbit of $\rho$ as $\rho_i = \alpha_i \beta_i$, then we get $f(\alpha_{i-1}) = \alpha_i \tau_i$ and $f(\beta_{i-1}) = \bar \tau_i \beta_i$, with the exception at the end where possibly (due to reversal of orientation)
$f(\alpha_m) = \bar \beta_0 \tau_0$ and $f(\beta_m) = \bar \tau_0 \bar \alpha_0$. Since $f$ is expanding, at least one of the paths $\tau_i$ is nontrivial.

We will now describe the process of folding an $f$-orbit of piNps. Given an $f$-orbit of piNps, $\{\alpha_0\beta_0, \ldots, \alpha_m\beta_m\}$, we know that at least one of the corresponding $\tau_i$ is nontrivial. Then we can fold at the turn between $\alpha_{i-1}$ and $\beta_{i-1}$. As more than one of the $\tau_i$ may be nontrivial, the full folds take precedence over the partial folds.

We say $f$ is {\bf stable} if the process of folding its $f$-orbits of piNps can always be done by full folds. If $f$ has no piNps, then it is vacuously stable. Since full folds do not increase the number of vertices in the graph, the process of folding a stable representative produces finitely many {\it projective (equivalence) classes} of graphs.

\begin{claim*} $\phi$ has a stable representative.\end{claim*}

Suppose $f$ was not stable. Then after doing some preliminary full folds, $f$ has an $f$-orbit of piNps whose only possible folds are partial folds. Fold this orbit enough times so that every turn ${\bar \alpha_i, \beta_i}$ is at a trivalent vertex. Now apply a homotopy so that the segments $\{\alpha_0\beta_0, \ldots, \alpha_m\beta_m\}$ are isometrically and cyclically permuted. Since $\phi$ is irreducible, these segments form an invariant forest which we can collapse to create a representative with strictly fewer piNps. As there were finitely many piNps to begin with, this process will terminate with a stable representative. This ends the claim.

\begin{claim*} A stable representative $f:\Gamma \to \Gamma$ has at most one $f$-orbit of piNps.\end{claim*}

Assume there is at least one $f$-orbit of piNps $\{\rho_i\}_{i=0}^m$. Let $\Gamma$ have a {\it Perron-Frobenius eigenmetric} and denote with $\vol(\Gamma)$ the sum of all the edge lengths.
Folding an $f$-orbit of piNps $\{\rho_i\}_{i=0}^m$ produces a graph $\Gamma'$ with $\vol(\Gamma') = \vol(\Gamma) - x$ and an $f'$-orbit of piNps $\{\rho_i'\}_{i=0}^m$ of $\Gamma'$ with $\vol(\{\rho_i'\}_{i=0}^m) = \vol(\{\rho_i\}_{i=0}^m) - 2x$ for some $x > 0$; the volume of a collection of paths is the sum of their lengths.
Since there are finitely many projective classes of graphs, the graph $\Gamma$ and $f$-orbit $\{\rho_i\}_{i=0}^m$ must satisfy the {\bf critical equation} $\vol\left(\{\rho_i\}_{i=0}^m\right) = 2  \vol(\Gamma)$. Note that this equation holds for any $f$-orbit of piNps.
Fix one such $f$-orbit $\{\rho_i\}_{i=0}^m$ and suppose there were an $f$-orbit $\{r_i\}_{i=0}^n$. If $\{\rho_i\}_{i=0}^m$ and $\{r_i\}_{i=0}^n$ did not share all their illegal turns, then folding $\{\rho_i\}_{i=0}^m$ would eventually decrease $\vol(\Gamma)$ while leaving $\vol(\{r_i\}_{i=0}^n)$ the same.
This would break the critical equation for $\{r_i\}_{i=0}^n$, contradicting stability. Therefore, all $f$-orbits of piNps have the same set of illegal turns. Since each fold in a stable representative is a full fold, there cannot be two distinct $f$-orbits of piNps that share the same illegal turns and maintain the critical equation throughout all the folds. This ends the second claim.

\begin{claim*} Stable representative $f:\Gamma \to \Gamma$ has exactly one $f$-orbit of piNps that make up an $f$-orbit of periodic Nielsen loops that collectively cover each edge of $\Gamma$ twice.\end{claim*}
The process of folding an $f$-orbit of piNps preserves periodic Nielsen loops. Since we started with a representative in which $\sigma$ is a periodic Nielsen loop, we know $f$ has a periodic Nielsen loop in $\Gamma$. In particular, the loop splits into piNps and therefore $f$ has exactly one $f$-orbit of piNps by the previous claim. This orbit makes up an $f$-orbit of periodic Nielsen loops containing $\sigma$. Let $\{s_i\}_{i=0}^n$  be an $f$-orbit of periodic Nielsen loops formed by concatenating paths in the $f$-orbit of piNps. To maintain the critical equation, folding $\{s_i\}_{i=0}^n$ must eventually reduce the lengths of all edges; therefore, every edge of $\Gamma$ appears at least once in $\{s_i\}_{i=0}^n$.

Suppose some edge of $\Gamma$ appeared exactly once in $\{s_i\}_{i=0}^n$, say in $s_n$. Then the loop $s_n$ determines a cyclic free factor $C_n$ of $F$ such that $F = B * C_n$ where all the other loops $s_i~(i \neq n)$ determine (conjugacy classes of) elements in $B$. As $\phi$ is injective, $s_1^{\pm1} = [f(s_n)]$ determines a cyclic free factor of $\phi(F) \cap B$;
since $s_n$ and $s_1$ can be simultaneously realized as free factors of $\phi(F)$ and $\phi$ is injective, the loops $s_{n-1}$ and $s_n$ determine cyclic free factors $C_{n-1}, C_n$ of $F$ that can be simultaneously realized, i.e.,  $F = B' * C_{n-1} * C_n$. Note that we use preimages to get the free factors of $F$ since we did not assume $\phi$ was an automorphism. Iterate this process to show $\{s_i\}_{i=0}^n$ determines a $\phi$-fixed cyclic free factor system $\{C_0, \ldots, C_n \}$. This contradicts the irreducibility assumption on $\phi$.
Thus every edge of $\Gamma$ appears at least twice in $\{s_i\}_{i=0}^n$ and the critical equation implies every edge appears exactly twice. This concludes the third claim and proof of the proposition.
\end{proof}

 The proof of the next theorem follows that of Bestvina-Handel \cite[Proposition~4.5]{BH92}. We give an outline with a modification that extends the argument to injective endomorphisms.

\begin{thm}\label{thmBH2} Suppose an endomorphism $\phi:F \to F$ is irreducible, it has infinite-order, and there exists $k \ge 1$ and nontrivial conjugacy class $[c]$ of $F$ such that $[\phi^k(c)] = [c]$.

Then $\phi$ is an automorphism induced by a pseudo-Anosov homeomorphism of a compact surface $\Sigma_g^{b \ge 1}$ that acts transitively on the boundary components.
\end{thm}
\begin{proof}[Sketch proof] Apply Proposition \ref{propStab} to get an irreducible train track $f:\Gamma \to \Gamma$ with an $f$-orbit of (unoriented) periodic indivisible Nielsen loops that collectively cover each edge of $\Gamma$ twice. Let $b \ge 1$ be the number of these periodic Nielsen loops. For each periodic Nielsen loop, attach an annulus by gluing one end along the loop and call the resulting space $M$; $M$ contains $\Gamma$ as a deformation retract, so $\pi_1(M) \cong F$.

Since the loops collectively cover each edge twice, the space $M$ is a surface except at finitely many singularities. Use the blow-up trick and the irreducibility of $\phi$ to conclude $M$ is a surface. Since $f$ transitively permutes the periodic Nielsen loops (up to homotopy, possibly reversing orientation), the map $f$ extends to a map $g:M \to M$ that transitively permutes the components of $\partial M$ such that $g_* = f_* = \phi$; so $g$ is $\pi_1$-injective.

Let $D(M)$ be the closed hyperbolic surface obtained by gluing two copies of $M$ along their boundary components. The map $g$ induces a $\pi_1$-injective map $g \cup_\partial g : D(M) \to D(M)$ such that $(g \cup_\partial g)_* = \phi *_\partial \phi$. But closed hyperbolic surfaces have coHopfian fundamental groups (classification of surfaces), therefore $\phi *_\partial \phi$ is an automorphism. This implies $\phi$ is an automorphism and $g$ is homotopic to a homeomorphism (Dehn-Nielsen-Baer theorem).

Assume $g$ is a homeomorphism. Any $g$-invariant collection of disjoint essential simple closed curves of $g$ determines a reduction of $\phi$, thus the irreducibility of $\phi$ implies $g$ is an infinite-order irreducible homeomorphism. By Nielsen-Thurston classification, the map $g$ is isotopic to a pseudo-Anosov homeomorphism of $M = \Sigma_g^b$ that acts transitively on the boundary components.
\end{proof}

\begin{cor}\label{irredIWIP} If $\phi:F \to F$ is an infinite-order irreducible endomorphism that is not fully irreducible, then $\phi$ is induced by a pseudo-Anosov homeomorphism of a compact surface $\Sigma_g^{b \ge 2}$ that acts transitively on the boundary components.
\end{cor}
\begin{proof} It follows from Proposition~\ref{invSbgrp} that for atoroidal endomorphisms, irreducible implies fully irreducible (see also \cite{DKL15, JPM}). So $\phi$ has a nontrivial periodic conjugacy class. By Theorem~\ref{thmBH2}, $\phi$ is induced by a pseudo-Anosov on a surface $\Sigma_g^{b \ge 1}$ that acts transitively on the boundary components. If $\phi$ is not fully irreducible, then $b \ge 2$ by Proposition~\ref{eg}.\end{proof}

%\nocite{*} % Insert publications even if they are not cited in the poster
\bibliography{refs}
\bibliographystyle{plain}
\vspace{-0.5em}
\end{document}